\documentclass[11pt,onecolumn,twoside]{IEEEtran}
\usepackage{amsmath,amssymb,euscript,yfonts,psfrag,latexsym,dsfont,graphicx,bbm,color,amstext,wasysym,pdfsync,textcomp}
\graphicspath{{figures/}}

\definecolor{gray}{RGB}{200,200,200}
\definecolor{red}{RGB}{200,10,10}

\newtheorem{thm}{Theorem}
\newtheorem{cor}[thm]{Corollary}

\newtheorem{prop}[thm]{Proposition}

\newtheorem{remark}[thm]{Remark}

\newcommand{\ignore}[1]{}

\newcommand{\mR}{{\mathbb R}}
\newcommand{\mD}{{\mathbb D}}

\newcommand{\mZ}{{\mathbb Z}}
\newcommand{\mC}{{\mathbb C}}

\newcommand{\cE}{{\mathcal E}}
\newcommand{\cM}{{\mathcal M}}
\newcommand{\cF}{{\mathcal F}}

\newcommand{\cD}{{\mathcal D}}

\newcommand{\cH}{{\mathcal H}}

\newcommand{\largedlbrack} {{\text{\Large{\textlbrackdbl}}}}
\newcommand{\largedrbrack}{{\text{\Large{\textrbrackdbl}}}}

\newcommand{\dlbrack} {{\text{{\textlbrackdbl}}}}
\newcommand{\drbrack}{{\text{{\textrbrackdbl}}}}

\newcommand{\bu}{{\bold u}}
\newcommand{\bg}{{\bold p}}
\newcommand{\bh}{{\bold h}}

\newcommand{\g}{{\operatorname{g}}}
\newcommand{\dist}{\operatorname{d}}
\newcommand{\D}{\operatorname{D}}

\newcommand{\intpi}{{\int_{-\pi}^\pi}}

\newcommand{\dtheta}{{\frac{d\theta}{2\pi}}}

\newcommand{\dmu}{{\frac{d\mu(\theta)}{2\pi}}}

\newcommand{\nhalf}{{-\frac{1}{2}}}

\newcommand{\jj}{{\rm j}}

\newcommand{\AR}{{\rm AR}}
\newcommand{\linear}{{\rm linear}}
\newcommand{\geodesic}{{\rm geodesic}}

\newcommand{\trace}{\operatorname{tr}}
\newcommand{\argmin}{\operatorname{argmin}}

\def\spacingset#1{\def\baselinestretch{#1}\small\normalsize}

\begin{document}
\title{Distances and Riemannian metrics\\ for multivariate spectral densities}
\author{Xianhua Jiang, Lipeng Ning, and Tryphon T. Georgiou,~\IEEEmembership{Fellow, IEEE}\thanks{Supported by NSF, AFOSR, and the Vincentine Hermes-Luh Endowment.
The authors are with the Department of Electrical \& Computer Engineering, University of Minnesota, Minneapolis, MN 55455.
{\{jiang082, ningx015, tryphon\}@umn.edu}}}
\date{\today}

\spacingset{1}
\markboth{\today}{\today}
\maketitle
\begin{abstract}
We first introduce a class of divergence measures between power spectral density matrices. These are derived by comparing the suitability of different models in the context of optimal prediction. 
Distances between ``infinitesimally close'' power spectra are quadratic, and hence, they induce a differential-geometric structure. We study the corresponding Riemannian metrics and, for a particular case, provide explicit formulae for the corresponding geodesics and geodesic distances.
The close connection between the geometry of power spectra and the geometry of the Fisher-Rao metric is noted.
\end{abstract}

\section{Introduction}
Distance measures between statistical models and between signals constitute some of the basic tools of Signal Processing, System Identification, and Control \cite{Basseville_distance,Gray_distortion}. Indeed, quantifying dissimilarities is the essence of detection, tracking, pattern recognition, model validation, signal classification, etc.
Naturally, a variety of choices are readily available for comparing deterministic signals and systems. These include various $L_p$ and Sobolev norms on signal spaces, and induced norms in spaces of systems. Statistical models on the other hand are not elements of a linear space. Their geometry is dictated by positivity constraints and hence, they lie on suitable cones or simplices. This is the case for covariances, histograms, probability distributions, or power spectra, as these need to be positive in a suitable sense. A classical theory for statistical models, having roots in the work of C.R. Rao and R.A. Fisher, is now known as ``information geometry'' \cite{Rao_information,Amari_methods,Cencov_statistical,Kass_geometrical}. The present work aims at a geometric theory suitable for time-series modeled by power spectra.
To this end, we follow a largely parallel route to that of information geometry (see  \cite{Georgiou_distance}) in that a metric is now dictated by the dissimilarity of models in the context of prediction theory for second-order stochastic processes.
The present work builds on \cite{Georgiou_distance}, which focused on scalar time-series, and is devoted to power spectral densities of multivariable stochastic processes.

The need to compare two power spectra densities $f_1,f_2$ directly has led to a number of divergence measures which have been suggested at various times \cite{Basseville_distance,Gray_distortion}. Key among those are the Itakura-Saito distance
\[
\D_\text{IS}(f_1,f_2):=\intpi \left( \frac{f_1(\theta)}{f_2(\theta)}-\log \frac{f_1(\theta)}{f_2(\theta)} -1\right)\dtheta
\]
and the logarithmic spectral deviation
\[
\D_{\log}(f_1,f_2):=\sqrt{\intpi \left|\log \frac{f_1(\theta)}{f_2(\theta)}\right|^2\dtheta},
\]
see e.g., \cite[page 370]{Gray_distortion}. The distance measures developed in \cite{Georgiou_distance} are closely related to both of these, and the development herein provides a multivariable counterpart. Indeed, the divergences that we list between matrix-valued power spectra are similar to the Itakura-Saito divergence and geodesics on the corresponding Riemannian manifolds of power spectra take the form of logarithmic integrals.

Distances between multivariable power spectra have only recently received any attention. In this direction we mention generalizations of the Hellinger and Itakura-Saito distances by Ferrante {\em et al.} \cite{Ferrante_time,Ferrante_hellinger} and the use of the Umegaki-von Neumann relative entropy \cite{Georgiou_relative}.
%As we will explain later on, these have connections to the non-commutative probability models of quantum mechanics and their corresponding geometry.
%
The goal of this paper is to generalize the geometric framework in \cite{Georgiou_distance} to the matrix-valued power spectra. We compare two power spectra in the context of linear prediction: a choice 
between the two is used to design an optimal filter which is then applied to a process corresponding to the second power spectrum. The ``flatness'' of the innovations process, as well as the degradation of the prediction error variance, when compared to the best possible, are used to quantify the mismatch between the two. This rationale provides us with natural divergence measures. We then identify corresponding Riemannian metrics that dictate the underlying geometry. For a certain case we compute closed-form expressions for the induced geodesics and geodesic distances.  These provide a multivariable counterpart to the logarithmic intervals in \cite{Georgiou_distance} and the logarithmic spectral deviation \cite[page 370]{Gray_distortion}.
It is noted that the geodesic distance has certain natural desirable properties; it is inverse-invariant and congruence-invariant. Moreover, the manifold of the multivariate spectral density functions endowed with this geodesic distance is a complete metric space. A discrete counter part of certain of these Riemannian metrics, on the manifold of positive definite matrices (equivalent to power spectra which are constant across frequencies), has been studied extensively in connection to the geometry of positive operators \cite{Bhatia_positive} and relates to the Rao-Fisher geometry on probability models restricted to the case of Gaussian random vectors.

Indeed, there is a deep connection between the Itakura-Saito distance and the Kullback-Leibler divergence between the corresponding probability models  \cite[page 371]{Gray_distortion}, \cite{Pinsker_information} which provides a link to information geometry. Hence, the Riemannian geometry on power spectral densities in \cite{Georgiou_distance} as well as the multivariable structure presented herein is expected to have a strong connection also to the Fisher-Rao metric and the geometry of information.  An interesting study in this direction which taps on an interpretation of the geometry of power spectra via the underlying probability structure and its connection to the Kullback-Leibler divergence is given in Yu and Mehta  \cite{Yu_kullback}. However, a transparent differential geometric explanation which highlights points of contact is still to be developed. Further key developments which parallel the framework reported herein and are focused on moment problems are presented in \cite{Ferrante_time,Ferrante_hellinger}.

The paper is organized as follows. In Section \ref{sec:prelim} we establish notation and overview the theory of the multivariate quadratic optimal prediction problem. In Section \ref{sec:III} we introduce alternative distance measures between multivariable power spectra which reflect mismatch in the context of one-step-ahead prediction.  In Section \ref{sec:Rie} we discuss  Riemannian metrics that are induced by the divergence measures of the previous section.  In Section \ref{sec:geometryonmatrices} we discuss the geometry of positive matrices. In Section \ref{sec:distances} the geometric structure is analyzed and geodesics are identified. In Section \ref{sec:example} we provide examples to highlight the nature of geodesics between power spectra and how these may compare to alternatives.

%%%%%%%%%%%%%%%%%%%%%%%%%%%%%%%%%%%
\section{Preliminaries on Multivariate Prediction}\label{sec:prelim}

Consider a multivariate discrete-time, zero mean, weakly stationary stochastic process $\{\bu(k),~ k\in\mZ\}$ with $\bu(k)$ taking values in $\mC^{m\times 1}$. Throughout, boldface denotes random variables/vectors, $\cE$ denotes expectation, $\jj=\sqrt{-1}$ the imaginary unit, and $^*$ the complex conjugate transpose.
Let
\[
R_k=\cE\left\{\bu(\ell) \bu^*(\ell-k)  \right\} ~\text{for } l,k\in\mZ
\]
denote the sequence of matrix covariances and $d\mu(\theta)$ be the corresponding matricial power spectral measure for which
\[
R_k=\intpi  e^{-\jj k\theta}\dmu.
\]
For the most part, we will be concerned with the case of non-deterministic processes with an absolutely continuous power spectrum. Hence, unless we specifically indicate otherwise,
$d\mu(\theta)=f(\theta)d\theta$ with  $f(\theta)$ being a matrix-valued power spectral density (PSD) function. Further, for a non-deterministic process $\log(f(\theta))$ needs to be integrable, and this will be assumed throughout as well.

Our interest is in comparing PSD's and in studying possible metrics between such. The evident goal is to provide a means to quantify deviations and uncertainty in the spectral domain in a way that is consistent with particular applications.
More specifically, we present metrizations of the space of PSD's which are dictated by optimal prediction and reflect dissimilarities that have an impact on the quality of prediction.

%%%%%%
\subsection{Geometry of multivariable processes}

We will be considering least-variance linear prediction problems. To this end, we define $L_{2,\bu}$ to be the closure of $m\times 1$-vector-valued finite linear combinations of
$\{\bu(k)\}$ with respect to covergence in the mean \cite[pg. 135]{Wiener_prediction}:
\[
L_{2,\bu}:=\overline{\left\{\sum_{\rm finite} P_k\bu(-k) \;:\; P_k\in\mC^{m\times m},\; k\in\mZ  \right\}}.
\]
Here, ``bar'' denotes closure.
The indices in $P_k$ and $\bu(-k)$ run in opposite directions so as to simplify the notation later on where prediction is based on past observations.
This space is endowed with both, a matricial inner product
\begin{displaymath}
\begin{split}
\largedlbrack \sum_k P_k\bu(-k), \sum_k Q_k \bu(-k) \largedrbrack :=\phantom{xxxxxxxxxxxxxxxxxx}\\
 \cE \left\{ \left( \sum_k P_k \bu(-k) \right) \left( \sum_k Q_k\bu(-k)\right)^* \right\},
\end{split}
\end{displaymath}
as well as a scalar inner product
\begin{displaymath}
\begin{split}
\langle \sum_k P_k\bu(-k), \sum_k Q_k \bu(-k) \rangle:=\phantom{xxxxxxxxxxxxxxxxxx}\\
\trace \largedlbrack \sum_k P_k\bu(-k), \sum_k Q_k \bu(-k) \largedrbrack.
\end{split}
\end{displaymath}
Throughout, ``$\trace$'' denotes the trace of a matrix.
It is standard to establish the correspondence between
\begin{align*}
\bg:=p(\bu)&:=\sum_k P_k \bu(-k) ~\text{and }\\
p(z)&:=\sum_k P_kz^k
\end{align*}
with $z=e^{\jj\theta}$ for $\theta\in[-\pi, \pi]$.
This is the Kolmogorov isomorphism between the ``temporal'' space $L_2(\bu)$ and ``spectral'' space $L_{2,d\mu}$,
\[
\varphi \;:\; L_2(\bu)\to L_{2,d\mu} \;:\; \sum_k P_k \bu(-k)\mapsto \sum_k P_kz^k.
\]
It is convenient to endow the latter space $L_{2,d\mu}$ with
the matricial inner product
\[
\dlbrack p, q\drbrack_{d\mu}:=\intpi \left(p(e^{\jj \theta})\dmu q(e^{\jj \theta})^*\right)
\]
as well as the scalar inner product
\[
\langle p, q\rangle_{d\mu}:=\trace \dlbrack p, q\drbrack_{d\mu}.\phantom{xxxxxxxx}
\]
The additional structure due to the matricial inner product is often referred to as {\em Hilbertian} (as opposed to {\em Hilbert})  \cite{Masani_RecentTrends}.

Throughout, $p(e^{\jj \theta})=\sum_k P_ke^{\jj k\theta}$, $q(e^{\jj \theta})=\sum_k Q_k e^{\jj k\theta}$, where we use lower case $p,q$ for matrix functions and upper case $P_k,Q_k$ for their matrix coefficients.
For non-deterministic processes with absolutely continuous spectral measure $d\mu(\theta)=f(\theta)d\theta$, we simplify the notation into
\begin{eqnarray*}
\dlbrack p, q\drbrack_{f}&:= & \dlbrack p, q\drbrack_{fd\theta}, \mbox{ and}\phantom{xxxx}\\
\langle p, q\rangle_{f} &:=& \langle p, q\rangle_{fd\theta}.
\end{eqnarray*}

Least-variance linear prediction
\begin{equation}\label{eq:predict}
    \min\left\{ \trace\cE\{\bg \bg^* \} : \bg=\bu(0)-\sum_{k>0}P_k\bu({-k}),\;P_k\in\mC^{m\times m} \right\}
\end{equation}
can be expressed equivalently in the spectral domain
\begin{equation}\label{eq:predictSpect}
\min \left\{ \dlbrack p,p\drbrack_f : p(z)=I-\sum_{k>0}P_kz^k,\;P_k\in\mC^{m\times m}  \right\}
\end{equation}
where the minimum is sought in the positive-definite sense, see  \cite[pg. 354]{Masani_RecentTrends}, \cite[pg. 143]{Wiener_prediction}. We use ``$I$'' to denote the identity matrix of suitable size.
It holds that, although non-negative definiteness defines only a partial order on the cone of non-negative definite Hermitian matrices, a minimizer for \eqref{eq:predict} always exists. Of course this corresponds to a minimizer for \eqref{eq:predictSpect}.
The existence of a minimizer is due to the fact that $\trace\cE\{\bg \bg^* \}$ is matrix-convex. Here $d\mu=fd\theta$ is an absolutely continuous measure and the
quadratic form is not degenerate; see \cite[Proposition 1]{Georgiou_CPD} for a detailed analysis and a treatment of the singular case where $\mu$ is a discrete matrix-valued measure.
Further, the minimizer of \eqref{eq:predict} coincides with the minimizer of
\begin{equation}\label{eq:predictSpect2}
\min \left\{ \langle p,p\rangle_f : p(z)=I-\sum_{k>0}P_kz^k,\;P_k\in\mC^{m\times m}  \right\}.
\end{equation}
From here on, to keep notation simple, $p(z)$ will denote the minimizer of such a problem, with $f$ specified accordingly, and
the minimal matrix of  \eqref{eq:predict} will be denoted by $\Omega$. That is,
\[
\Omega:= \dlbrack p,p\drbrack_f
\]
while the minimal value of \eqref{eq:predictSpect2} is $\trace{\Omega}$.
The minimizer $p$ is precisely the image under the Kolmogorov isomorphism of the optimal {\em prediction error} $\bg$ and $\Omega$ the {\em prediction-error variance}.

\subsection{Spectral factors and optimal prediction}
For a non-deterministic process the error variance $\Omega$ has full rank.
Equivalently, the product of its eigenvalues is non-zero.
The well-known Szeg\"{o}-Kolmogorov formula \cite[pg. 369]{Masani_RecentTrends}
\begin{equation}\label{Szego_Kolmogorov}
\det \Omega = \exp\{\intpi \log\det f(\theta) \dtheta\}
\end{equation}
relates the product of the eigenvalues of the optimal one-step-ahead prediction error variance with the corresponding PSD.
No expression is available in general that would relate $f$ to $\Omega$ directly in the matricial case.

We consider only non-deterministic processes and hence we assume that
\[\log \det f(\theta)\in L_1[-\pi, \pi].\]
In this case, $f(\theta)$ admits a unique factorization
\begin{align}\label{SpectralFactorize}
f(\theta)=f_+(e^{\jj\theta})f_+(e^{\jj\theta})^*,
\end{align}
with $f_+(e^{\jj\theta})\in \cH_2^{m\times m}(\mD)$,
\[
\det(f_+(z))\neq 0 \mbox{ in } {\mathbb D}:=\{z: |z|<1\},
\]
and normalized so that $f_+(0)=\Omega^\frac12$. Throughout, $M^\frac12$ denotes the Hermitian square root of a Hermitian matrix $M$. The factor $f_+$ is known as the  {\em canonical (left) spectral factor}.
In the case where $f$ is a scalar function ($m=1$) the canonical spectral factor is explicitly given by
\[
f_+(z)=\exp\left\{  \frac12 \intpi \left(\frac{1+ze^{-\jj\theta}}{1-ze^{-\jj\theta}}\right)\log f(\theta)\dtheta \right\},~~|z|<1,
\]
%In the multivariable case, an infinite series expansion has been derived under the condition that the eigenvalues of $f(\theta)$ are bounded \cite[pg. 124]{Wiener_predictionII} which applies to general power spectra (not necessarily rational).
As usual, $\cH_2(\mD)$ denotes the Hardy space of functions which are analytic in the unit disk $\mD$ with square-integrable radial limits. Spectral factorization presents an ``explicit'' expression of
the optimal prediction error in the form
\begin{equation}\label{eq:spectral_factor}
p(z)=f_+(0)f_+^{-1}(z).
\end{equation}
Thus, $p(z)^{-1}$ is a ``normalized'' (left) {\em outer} factor of $f$.  The terminology ``outer'' refers to a (matrix-valued) function $g(e^{\jj\theta})$ for $\theta\in[-\pi,\pi]$ that can be extended into an analytic function in the open interior of the unit disc $\mD$ which is also invertible in $\mD$. 
It is often standard not to differentiate between such a function in $\mD$ and the function on the boundary of radial-limits since these are uniquely defined from one another. In the engineering literature outer functions are also referred to as ``minimum phase.'' Right-outer factors, where $f(\theta)=f_{+,\rm right}(e^{\jj\theta})^*f_{+,\rm right}(e^{\jj\theta})$ instead of \eqref{SpectralFactorize}
relate to a {\em post-diction} optimal estimation problem; in this, the present value of the process is estimated via linear combination of future values (see e.g., \cite{Georgiou_CPD}). Only left factorizations will be used in the present paper.

\section{Comparison of PSD's}\label{sec:III}

We present two complementing viewpoints on how to compare two PSD's, $f_1$ and $f_2$.
In both, the optimal one-step-ahead predictor for one of the two stochastic processes, is applied to the other and compared to the corresponding optimal.
The first is to consider how ``white'' the power spectrum of the innovations' process is.
The second viewpoint is to compare how the error variance degrades with respect to the optimal predictor. Either principle provides a family of divergence measures and a suitable generalization of the Riemannian geometry of scalar PSD's given in \cite{Georgiou_distance}.
There is a close relationship between the two.

\subsection{Prediction errors and innovations processes}

Consider two matrix-valued spectral density functions $f_1$ and $f_2$.
Since an optimal filter will be designed based on one of the two and then evaluated with respect to the other, some notation is in order.

First, let us use a subscript to distinguish between two processes $\bu_i(k)$, $i\in\{1,2\}$, having the $f_i$'s as the corresponding PSD's. They are assumed purely nondeterministic, vector-valued, and of compatible size. The optimal filters in the spectral domain are
\begin{eqnarray*}
p_i&:=&\argmin\{\largedlbrack p,p\largedrbrack_{f_i} \;\:\; p(0)=I,\\
&&\phantom{ixxxxx}\mbox{ and } p\in \cH_2^{m\times m}(\mD)\},
\end{eqnarray*}
and their respective error covariances
\begin{eqnarray*}
\Omega_i&:=&\largedlbrack p_i,p_i\largedrbrack_{f_i}.
\end{eqnarray*}
Now define
\begin{eqnarray*}
\Omega_{i,j}&:=&\largedlbrack p_j,p_j\largedrbrack_{f_i}.
\end{eqnarray*}
Clearly, $\Omega_{i,j}$ is the variance of the prediction error when the filter $p_j$ is used on a process having power spectrum $f_i$. Indeed, if we set
\begin{align}
\bg_{i,j}:= \bu_{i}(0)-P_{j,1}\bu_i(-1)-P_{j,2}\bu_i(-2)-\ldots  \label{eq:innovation0}
\end{align}
the prediction-error covariance is
\[ \dlbrack  \bg_{i,j},  \bg_{i,j}\drbrack=\dlbrack p_j, p_j \drbrack_{f_i}.
\]

The prediction error $\bg_{i,j}$ can also be thought of as a time-process, indexed at time-instant $k\in\mZ$,
\begin{equation}\label{eq:innovation1}\nonumber
\bg_{ij}(k):= \bu_i(k)-P_{j,1}\bu_i(k-1)-P_{j,2}\bu_i(k-2)-\ldots
\end{equation}
for $i,j\in\{1,2\}$. This is an {\em innovations process}. Clearly, from stationarity,
\[
\dlbrack  \bg_{i,i},  \bg_{i,i}\drbrack=\Omega_{i},
\]
whereas
\[
\dlbrack  \bg_{i,j},  \bg_{i,j}\drbrack\geq\Omega_{i},
\]
since in this case $p_j$ is suboptimal for $\bu_i$, in general.

\subsection{The color of innovations and PSD mismatch}

We choose to normalize the innovations processes as follows:
\begin{equation}\label{eq:innovation2}\nonumber
\bh_{i,j}(k)=\Omega_{j}^{\nhalf}\bg_{i,j}(k), \mbox{ for } k\in\mZ.
\end{equation}
The Kolmogorov isomorphism takes
\[
\varphi\;:\; \bh_{i,j}(k) \mapsto f_{j+}^{-1},
\]
with the expectation/inner-product being that induced by $f_i$, and hence, the power spectral density of the process $\bh_{i,j}(k)$ is
\[
f_{\bh_{ij}}=f_{j+}^{-1}f_i f_{j+}^{-*},
\]
where $(\cdot)^{-*}$ is a shorthand for $((\cdot)^*)^{-1}$.
When $f_i=f_j$, evidently $\{\bh_{k}^{i,i}\}$ is a white noise process with covariance matrix equals to the identity.

Naturally, in an absolute sense, the mismatch between the two power spectra $f_i,f_j$ can be quantified by the distance of $f_{\bh_{ij}}$ to the identity.
To this end we may consider any symmetrized expression:
\begin{equation}\label{eq:classdistance}
\intpi \dist(f_{j+}^{-1}f_i f_{j+}^{-*}, I)  \dtheta+\intpi \dist(f_{i+}^{-1}f_j f_{i+}^{-*}, I)  \dtheta
\end{equation}
for a suitable distance $\dist(\cdot,\cdot)$ between positive definite matrices.
In general, it is deemed desirable that distances between power spectra are invariant to scaling (as is the case when distances depend on ratios of spectra, \cite{Gray_distortion}). Researchers and practitioners alike have insisted on such a property, especially for speech and image systems, due to an apparent agreement with subjective qualities of sound and images. It is thus interesting to seek a multivariable analogues inherent in the above comparison.

Due to the non-negative definiteness of power spectra, a convenient option is to take ``$\dist$'' as the trace:
\begin{eqnarray}\nonumber
\intpi \hspace*{-7pt}\trace\left( f_{j+}^{-1}f_i f_{j+}^{-*}-I\right)+
\trace\left( f_{i+}^{-1}f_j f_{i+}^{-*}-I\right)\dtheta.
\end{eqnarray}
This indeed defines a distance measure since $(x+x^{-1}-2)$ is a non-negative function for $0<x\in\mR$ that vanishes only when $x=1$.
Thus, we define
\begin{subequations}\label{D2}
\begin{equation}
\D_1(f_1, f_2):=\intpi \hspace*{-5pt}\trace\left( f_{2}^{-1}f_1+f_{1}^{-1}f_2-2I\right)\dtheta.
\label{DistDef1}%\label{D2a}
\end{equation}

Interestingly, $\D_1(f_1, f_2)$ can be re-written as follows:
\begin{equation}\label{D2b}
\D_1(f_1, f_2)=\intpi \| f_1^{-1/2}f_2^{1/2}- {f_1^{1/2}}{f_2^{-1/2}}\|_{\rm Fr}^2 \dtheta
\end{equation}
where $\|M\|_{\rm Fr}^2:=\trace MM^*$
denotes the square of the Frobenius norm\footnote{$\sqrt{\trace MM^*}$ is also referred to also as the Hilbert-Schmidt norm.}.
\end{subequations}
It can be readily verified starting from the right hand side of \eqref{D2b} and simplifying this to match \eqref{DistDef1}.
It is now be easily seen that $\D_1(f_i,f_j)$ has a number of desirable properties listed in the following proposition.
\begin{prop} Consider $f_i,f_j$ being PSD's of non-deterministic processes and $g(e^{\jj\theta})$ an arbitrary outer matrix-valued function in $ \cH^{m\times m}_2(\mD)$.
The following hold:
\begin{itemize}
\item[(i)] $\D_1(f_i,f_j)\geq 0$.
\item[(ii)] $\D_1(f_i,f_j)= 0$ iff $f_i=f_j$ (a.e.).
\item[(iii)] $\D_1(f_i,f_j)=\D_1(f_j,f_i)$.
\item[(iv)] $\D_1(f_i,f_j)=\D_1(f_i^{-1},f_j^{-1})$.
\item[(v)] $\D_1(f_i,f_j)=\D_1(gf_ig^*,gf_jg^*)$.
\end{itemize}
\end{prop}

\vspace*{.1in}
\begin{proof}
Properties (i-iv) follow immediately from \eqref{D2b} while the invariance property (v) is most easily seen by employing \eqref{DistDef1}.
\end{proof}

\subsection{Suboptimal prediction and PSD mismatch}\label{sec:deg}

We now attempt to quantify how suboptimal the performance of a filter is when this is based on the incorrect choice between the two alternative PSD's. To this end, we consider the error covariance and compare it to that of the optimal predictor. A basic inequality between these error covariances is summarized in the following proposition.

\begin{subequations}\label{}

\begin{prop}\label{prop:inequalities} Under our earlier standard assumptions, for $i,j\in\{1,2\}$ and $\Omega_i, \Omega_j>0$, it holds that
\begin{equation}\label{first_inequality}
\Omega_{i,j}\geq \Omega_i.
\end{equation}
Further, the above holds as an equality iff $p_i=p_j$.
%the following hold:
%\begin{itemize}
%\item[(i$_a$)] $\Omega_{i,j}\geq \Omega_i$, while this holds with equality iff $p_i=p_j$.
%\item[(i$_b$)] $\Omega_i^{-\frac12}\Omega_{i,j}\Omega_i^{-\frac12}\geq I$, with equality iff $p_i=p_j$.
%\item[(ii)] $\det(\Omega_{i,j})\geq \det(\Omega_i)$, with equality iff $p_i=p_j$.
%\item[(iii)] $\|\Omega_j^{-\frac12}\Omega_{i,j}\Omega_j^{-\frac12}- I\|\geq 0$.
%\end{itemize}
\end{prop}

\vspace*{.1in}
\begin{proof} It follows from the optimality of $p_i$ since
\[ \dlbrack p_j, p_j \drbrack_{f_i} \geq \dlbrack p_i, p_i \drbrack_{f_i}= \Omega_i.\]
%, while (i$_b$) is essentially the same inequality. Inequality (ii) follows immediately from (i) and the fact that $\Omega_i>0$.
%The claim in (iii) follows by comparing
%\begin{eqnarray*}
%\Omega_j^{-\frac12}\Omega_{i,j}\Omega_j^{-\frac12} &=& \Omega_j^{-\frac12}\largedlbrack p_i,p_i\largedrbrack_{f_i}\Omega_j^{-\frac12}\\
%&=& \intpi (f_{j+}(e^{\jj\theta}))^{-1}f_i(\theta)(f_{j+}(e^{\jj\theta})^*)^{-1}\dtheta
%\end{eqnarray*}
%to the identity matrix.
\end{proof}

\begin{cor} The following hold:
\begin{eqnarray}
\Omega_i^{-\frac12}\Omega_{i,j}\Omega_i^{-\frac12}&\geq& I \label{second}\\
\det(\Omega_{i,j})&\geq& \det(\Omega_i) \label{third}\\
\trace(\Omega_{i,j})&\geq& \trace(\Omega_i)\label{forth}\\
\Omega_j^{-\frac12}\Omega_{i,j}\Omega_j^{-\frac12}&\geq& \Omega_j^{-\frac12}\Omega_i\Omega_j^{-\frac12}.\label{fifth}
\end{eqnarray}
Further, each ``$\geq$'' holds as equality iff $p_i=p_j$.
\end{cor}
\end{subequations}

Thus, a mismatch between the two spectral densities can be quantified by the strength of the above inequalities. To this end, we may consider a number of alternative ``divergence measures''. First we consider:
\begin{eqnarray}
\D_2(f_i,f_j)&:=& \log\det\left(\Omega_i^{-\frac12}\Omega_{i,j}\Omega_i^{-\frac12}\right). \label{D1}
\end{eqnarray}
Equivalent options leading to the same Riemannian structure are:
\begin{subequations}\label{eq:options}
\begin{eqnarray}
&&\frac{1}{m}\trace(\Omega_i^{-\frac12}\Omega_{i,j}\Omega_i^{-\frac12})-1 \label{D2aa}, \mbox{ and}\\
&&\det(\Omega_i^{-\frac12}\Omega_{i,j}\Omega_i^{-\frac12})-1.
\end{eqnarray}
\end{subequations}

Using the generalized Szeg\"{o}-Kolmogorov expression \eqref{Szego_Kolmogorov} we readily obtain that
%\[
%\det \Omega_i=\exp\left\{ \intpi\log \det f_i\dtheta \right\}~~\text{for~} i=1,~2,
%\]
%we have
%\[
%\det \Omega_1^{-1}\det \Omega_2=\exp\left\{ \intpi\log \det (f_1^{-1}f_2)\dtheta \right\}.
%\]
%On the other hand,
%\[
%\det \Omega_1^{\nhalf} \dlbrack p_1, p_1 \drbrack_{f_2} \Omega_1^{\nhalf}=\det\left( \intpi f_{1+}^{-1}f_2f_{1+}^{-*}\dtheta\right).
%\]
%Thus
%\[
%\rho(f_1,f_2) =\frac{\det \left(\intpi f_{1+}^{-1}f_2f_{1+}^{-*}\dtheta\right)}{\exp\left\{ \intpi\log \det (f_1^{-1}f_2)\dtheta \right\}}.
%\]
%Since $\rho_{W}(f_1,f_2)\geq 1$, the following expressions
%\begin{equation}\label{dw}
%d(f_1,f_2) := \rho_W(f_1,f_2)-1.
%\end{equation}
\begin{eqnarray}\label{deltap}
\D_2(f_i,f_j) &=&\log \det\!\left( \!\intpi f_{j+}^{-1}f_if_{j+}^{-*}\dtheta\!\right)\!-\!\intpi\log \det \! \left(\!f_{j+}^{-1}f_if_{j+}^{-*}\!\right)\dtheta\\\nonumber
&=&\trace \left(\log \intpi f_{j+}^{-1}f_if_{j+}^{-*}\dtheta-\intpi \log f_{j+}^{-1}f_if_{j+}^{-*}\dtheta\right)\nonumber.
\end{eqnarray}
This expression takes values in $[0, \infty]$, and is zero if and only if the normalized spectral factors $p^{-1}=\Omega^{-1/2}f_+$ are identical for the two spectra.
Further, it provides a natural generalization of the divergence measures in \cite{Georgiou_distance} and of the Itakura distance to the case of multivariable spectra. It satisfies ``congruence invariance.'' This is stated next.

\begin{prop}\label{prop:g_outer}
Consider two PSD's $f_i,f_j$ of non-deterministic processes and $g(e^{\jj\theta})$ an outer matrix-valued function in $ \cH^{m\times m}_2(\mD)$.
The following hold:
\begin{itemize}
\item[(i)] $\D_2(f_i,f_j)\geq 0$.
\item[(ii)] $\D_2(f_i,f_j)= 0$ iff $p_i=p_j$.
\item[(iii)] $\D_2(f_i,f_j)=\D_2(gf_ig^*,gf_jg^*)$.
\end{itemize}
\end{prop}

\vspace*{.1in}
\begin{proof}
Properties (i-ii) follow immediately from \eqref{D1} while the invariance property (iii) is most easily seen be employing \eqref{deltap}.
To this end, first note that $gf_+$ obviously constitutes the spectral factor of $gfg^*$. Substituting the corresponding expressions in \eqref{deltap} establishes the invariance.
\end{proof}

\subsection{Alternative divergence measures}

Obviously, a large family of divergence measures between two matrix-valued power spectra can be obtained based on (\ref{eq:classdistance}). For completeness, we suggest representative possibilities some of which have been independently considered in recent literature.

\begin{subequations}\label{DistFro}
\subsubsection{Frobenius distance}
If we use the Frobenius norm in \eqref{eq:classdistance} we obtain
  \begin{align}\label{eq:DistFro}
  \D_\text{F}(f_1, f_2):={\frac12}\sum_{i,j} \intpi \|f_{j+}^{-1}f_i f_{j+}^{-*}- I\|_{\rm Fr}^2  \dtheta
  \end{align}
  where $\sum_{i,j}$ designates the ``symmetrized sum'' taking $(i,j)\in\{(1,2),(2,1)\}$.
It's straightforward to see that all of
\[
f_{j+}^{-1}f_i f_{j+}^{-*},\;f_{j}^{\nhalf}f_i f_{j}^{\nhalf}\mbox{ and }f_{j}^{-1}f_i
\]
share the same eigenvalues for any $\theta\in[-\pi,\pi]$. Thus,
%the eigenvalues  $\lambda_l(f_{j+}^{-1}f_i f_{j+}^{-*})$, $l=1, \cdots, m$, have the following property:
%  \[
%  \lambda_l(f_{j+}^{-1}f_i f_{j+}^{-*})= \lambda_l(f_{j}^{-1}f_i )= \lambda_l(f_{j}^{\nhalf}f_i f_{j}^{\nhalf}).
%  \]
%  Therefore,
  \[
  \|f_{j+}^{-1}f_i f_{j+}^{-*}- I\|_{\rm Fr}^2=\|f_{j}^{\nhalf}f_i f_{j}^{\nhalf}- I\|_{\rm Fr}^2,
  \]
  and
  \begin{equation}\label{DF2}
   \D_\text{F}(f_1, f_2)={\frac12}\sum_{i,j} \intpi  \|f_{j}^{\nhalf}f_i f_{j}^{\nhalf}- I\|_{\rm Fr}^2 \dtheta .
  \end{equation}
 Obviously \eqref{DF2} is preferable over \eqref{eq:DistFro} since no spectral factorization is involved.
 \end{subequations}

\subsubsection{Hellinger distance}

A generalization of the Hellinger distance has been proposed in \cite{Ferrante_hellinger} for comparing multivariable spectra.
Briefly, given two positive definite matrices $f_1$ and $f_2$
one seeks factorizations $f_i=g_ig_i^*$ so that the integral over frequencies of the Frobenius distance $\|g_1-g_2\|_{\rm Fr}^2$ between the factors is minimal.
The factorization does not need to correspond to analytic factors.
%Thus, one may take 
%\[
%g_1=f_1^{\frac{1}{2}} \mbox{ and }g_2=f_2^{\frac{1}{2}}U
%\]
%with $U$ a suitable frequency-dependent unitary factor to maximize the real part
%\[
%\Ree \trace(f_1^{\frac{1}{2}} U f_2^{\frac{1}{2}}),
%\]
%since $\|g_1-g_2\|_{\rm Fr}^2=f_1+f_2-2\Ree \trace(f_1^{\frac{1}{2}} U f_2^{\frac{1}{2}})$.
%It turns out that the optimal unitary transformation is
%\[
%U=f_2^{\nhalf}f_1^{\nhalf}(f_1^{1/2}f_2f_1^{1/2})^{1/2}.
%\]
When one of the two spectra is the identity, the optimization is trivial and the Hellinger distance becomes
\[\intpi \|f^\frac12 - I\|_{\rm Fr}^2\dtheta.
\]
A variation of this idea is to compare the normalized innovation spectra 
$(f_{j+}^{-1}f_i f_{j+}^{-*})^{\frac12}$, for $i,j\in\{1,2\}$, to the identity. We do this in a symmetrized fashion so that together with symmetry the metric inherits the inverse-invariance property. Thus, we define
  \begin{align}\label{eq:DistHellinger}
    \D_\text{H}(f_1, f_2)\!&:=
 \!\!\sum_{i,j} \intpi \|(f_{j+}^{-1}f_i f_{j+}^{-*})^{\frac12}- I\|_{\rm Fr}^2  \dtheta \\ & =
  \! \sum_{i,j} \intpi \|(f_{j}^{\nhalf}f_i f_{j}^{\nhalf})^{\frac12}- I\|_{\rm Fr}^2  \dtheta .\nonumber
    \end{align}
The second equality follows by the fact that $f_{j+}f_{j}^{-\frac12}$ is a frequency-dependent unitary matrix.

\subsubsection{Multivariable Itakura-Saito distance}
The classical Itakura-Saito distance can be readily generalized by taking
\[
      \dist(f, I)=\trace (f -\log f-I).
\]
The values are always positive for $I\neq f>0$ and equal to zero when $f=I$.  Thus, we may define
     \begin{align}\label{eq:DistItakura}
      \D_{\text{IS}}(f_1, f_2)&=\intpi \dist(f_{2+}^{-1}f_1 f_{2+}^{-*}, I) \dtheta \\
      &=\intpi \left( \trace(f_{2}^{-1}f_1)-\log\det(f_{2}^{-1}f_1)-m \right) \dtheta.\nonumber
     \end{align}
The Itakura-Saito distance has its origins in maximum likelihood estimation for speech processing and is related to the Kullback-Leibler divergence between the probability laws of two Gaussian random processes \cite{Gray_distortion,Pinsker_information}. More recently, \cite{Ferrante_time} introduced the matrix-version of the Itakura-Saito distance for solving the state-covariance matching problem in a multivariable setting.

\subsubsection{Log-spectral deviation}
It has been argued that a logarithmic measure 
of spectral deviations is in agreement with perceptive qualities of sound and for this reason it has formed the basis for the oldest distortion measures considered \cite{Gray_distortion}. In particular, the $L_2$ distance between the logarithms of power spectra is referred to as ``Log-spectral deviation'' or the ``logarithmic energy.''  A natural multivariable version
is to consider
\[ \dist(f,I)=\|\log(f)\|^2_{\rm Fr}.\]
This expression is already symmetrized, since $\dist(f,I)=\dist(f^{-1},I)$ by virtue of the fact that the eigenvalues of $\log(f)$ and those of $\log(f^{-1})$ differ only in their sign. Thereby,
  \[
  \| \log(f_{j+}^{-1}f_i f_{j+}^{-*}) \|_{\rm Fr}^2=\| \log(f_{i+}^{-1}f_j f_{i+}^{-*}) \|_{\rm Fr}^2.
  \]
Thus we define
  \begin{align}\label{eq:DistLog}
 \D_{\text{Log}}(f_1, f_2)&:=\intpi \|\log(f_{1+}^{-1}f_2 f_{1+}^{-*})\|_{\rm Fr}^2  \dtheta \\
  &=\intpi \|\log(f_{1}^{\nhalf}f_2 f_{1}^{\nhalf})\|_{\rm Fr}^2  \dtheta .\nonumber
  \end{align}
This represents a multivariable version of the log-spectral deviation (see \cite[page 370]{Gray_distortion}). Interestingly, as we will see later on, $\D_\text{Log}(f_1, f_2)$  possesses several useful properties and, in fact, its square root turns out to be precisely a geodesic distance in a suitable Riemannian geometry.

\section{Riemannian structure on multivariate spectra}\label{sec:Rie}

Consider a ``small'' perturbation $f+\Delta$ away from a nominal power spectral density $f$.
All divergence measures that we have seen so far are continuous in their arguments and, in-the-small, can be approximated by a quadratic form in $\Delta$ which depends continuously on $f$. This is what is referred to as a {\em Riemannian metric}. The availability of a metric gives the space of power spectral densities its properties. It dictates how perturbations in various directions compare to each other. It also provides additional important concepts: geodesics, geodesic distances, and curvature. Geodesics are paths of smallest length connecting the start to the finish; this length is the geodesic distance. Thus, geodesics in the space of power spectral densities represent deformations from a starting power spectral density $f_0$ to an end ``point'' $f_1$. Curvature on the other hand is intimately connected with approximation and convexity of sets.

In contrast to a general divergence measure, the geodesic distance obeys the triangular inequality and thus, it is a metric (or, a pseudo-metric when by design it is unaffected by scaling or other group of transformations). Geodesics are also natural structures for modeling changes and deformations. In fact, a key motivation behind the present work is to model time-varying spectra via geodesic paths in a suitable metric space. This viewpoint provides a non-parametric model for non-stationary spectra, analogous to a spectrogram, but one which takes into account the inherent geometry of power spectral densities.

Thus, in the sequel we consider infinitesimal perturbations about a given power spectral density function. We explain how these give rise to nonnegative definite quadratic forms. Throughout, we assume that all functions are smooth enough so that the indicated integrals exist. This can be ensured if all spectral density functions are bounded with bounded derivatives and inverses. Thus, we will restrict our attention to the following class of PDF's:
\begin{eqnarray*}
 \cF&:=&\{ f\;\mid \mbox{$m\times m$ positive definite, differentiable}\\
&&\phantom{x}\mbox{on }[-\pi,\pi], \mbox{ with continuous derivative}\}.
\end{eqnarray*}
In the above, we identify the end points of $[-\pi,\pi]$ since $f$ is thought of as a function on the unit circle.
Since the functions $f$ are strictly positive definite and bounded, tangent directions of $\cF$ consists of admissible perturbations $\Delta$. These need only be restricted to be differentiable with square integrable derivative, hence the tangent space at any $f\in\cF$ can be identified with
\begin{eqnarray*}
\cD&:=&\{\Delta \;\mid  \mbox{differentiable on }[-\pi,\pi]\\
&&\mbox{ with continuous derivative}\}.
\end{eqnarray*}

\subsection{Geometry based on the  ``flatness'' of innovations spectra}

\begin{subequations}
We first consider the divergence $\D_1$ in (\ref{DistDef1}-\ref{D2b}) which quantifies how far the PSD of the normalized innovations process is from being constant and equal to the identity.
The induced Riemannian metric takes the form
\begin{align}\label{MetricDef1}
\g_{1,f}(\Delta):=\intpi \|f^{-1/2}\Delta f^{-1/2} \|_{\rm Fr}^2 \dtheta.
\end{align}

\begin{prop}\label{Prop:Metric1}
Let $(f,\Delta) \in\cF\times\cD$ and $\epsilon>0$. Then, for $\epsilon$ sufficiently small,
\begin{eqnarray*}
\D_1(f,f+\epsilon\Delta)=\g_{1,f}(\epsilon\Delta)+O(\epsilon^3).
\end{eqnarray*}
\end{prop}
\vspace*{.1in}

\begin{proof}
First note that
\begin{align*}
\trace \left( f(f+\epsilon\Delta)^{-1}\right)&=\trace \left(f^{1/2}(I+f^{-1/2}\epsilon\Delta f^{-1/2} )^{-1}f^{-1/2}\right)\\
&=\trace \left(I+f^{-1/2}\epsilon\Delta f^{-1/2} \right)^{-1}\\
\trace \left( f(f+\epsilon\Delta)^{-1}\right)&=m-\trace (f^{-1/2}\epsilon \Delta f^{-1/2})\\
&~+\|f^{-1/2}\epsilon\Delta f^{-1/2} \|_{\rm Fr}^2+O(\epsilon^3).
\end{align*}
%Let $\lambda_i$,~$i=1, \cdots, m$ be eigenvalues of $f^{-1/2}\Delta f^{-1/2}$, then
%\begin{align*}
%\trace \left(I+f^{-1/2}\epsilon\Delta f^{-1/2} \right)^{-1}&=\sum_{i=1}^m (1+\epsilon\lambda_i)^{-1}\\
%&=\sum_{i=1}^m (1-\epsilon\lambda_i+\epsilon^2\lambda_i^2+O(\epsilon^3))\\
%%&=m-\trace (f^{-1/2}\epsilon \Delta f^{-1/2})+\|f^{-1/2}\Delta f^{-1/2} \|_{\rm Fr}^2+O(\epsilon^3),
%\end{align*}
%where in the second equality we have used the Taylor series and the assumption that $|\lambda_i|<1$ when dealing with the higher order terms.
%Since
%\begin{align*}
%\sum_{i=1}^m \epsilon\lambda_i&=\trace ( f^{-1/2}\epsilon \Delta f^{-1/2})
%\end{align*}
%and
%\begin{align*}
%\sum_{i=1}^m \epsilon^2\lambda_i^2&=\|f^{-1/2}\epsilon\Delta f^{-1/2} \|_{\rm Fr}^2,
%\end{align*}
%we have
%\begin{align*}
%\trace \left( f(f+\epsilon\Delta)^{-1}\right)&=m-\trace (f^{-1/2}\epsilon \Delta f^{-1/2})\\
%&~+\|f^{-1/2}\epsilon\Delta f^{-1/2} \|_{\rm Fr}^2+O(\epsilon^3).
%\end{align*}
Likewise,
\begin{align*}
\trace(f+\epsilon\Delta)f^{-1}&=m+\trace(\epsilon\Delta f^{-1})\\
&=m+\trace (f^{-1/2}\epsilon \Delta f^{-1/2}).
\end{align*}
Therefore,
\begin{align*}
\D_1(f, f\!+\!\epsilon\Delta)&=\trace \!\intpi \!\!\left(f(f+\epsilon\Delta)^{-1}\!+\! (f+\epsilon\Delta)f^{-1}\!-\!2I\right)\dtheta\\
&=\intpi \|f^{-1/2}\epsilon\Delta f^{-1/2} \|_{\rm Fr}^2 \dtheta+ O(\epsilon^3).
\end{align*}
\end{proof}

Obviously, an alternative expression for $\g_{1,f}$ that requires neither spectral factorization nor the computation of the Hermitian square root of $f$, is the following: 
\begin{align}\label{MetricDef1b}
\g_{1,f}(\Delta):=\intpi \trace\left(f^{-1}\Delta f^{-1} \Delta\right) \dtheta.
\end{align}
It is interesting to also note that any of  (\ref{DistFro}), (\ref{eq:DistHellinger}), (\ref{eq:DistItakura}), and (\ref{eq:DistLog}) leads to the same Riemannian metric.
\end{subequations}

\subsection{Geometry based on suboptimality of prediction}

The paradigm in \cite{Georgiou_distance}
for a Riemannian structure of scalar power spectral densities was originally built on the degradation of predictive error variance, as this is reflected in the strength of the inequalities of Proposition \ref{prop:inequalities}. In this section we explore the direct generalization of that route. Thus, we consider the quadratic form which $\cF$ inherits from the relevant divergence $\D_2$, defined in \eqref{D1}. The next proposition shows that this defines the corresponding metric:
%\begin{subequations}\label{eq:g1}
\begin{eqnarray}\nonumber
\g_{2,f}(\Delta)&:=&\hspace*{-2pt}\trace\intpi \hspace*{-2pt}(f^{-1}_+\Delta f^{-*}_+)^2\dtheta-\trace\big(\intpi \hspace*{-2pt}f^{-1}_+\Delta f^{-*}_+\dtheta\big)^2\\\label{g1a}
&=&\hspace*{-2pt}\g_{1,f}(\Delta)-\trace\big(\intpi f^{-1}_+\Delta f^{-*}_+\dtheta\big)^2.
\end{eqnarray}

\begin{prop}\label{prop:dfDeltaf}
Let $(f,\Delta) \in\cF\times\cD$ and $\epsilon>0$. Then, for $\epsilon$ sufficiently small,
\begin{align*}
\D_2(f,f+\epsilon\Delta)=\frac{1}{2} \g_{2,f}(\epsilon\Delta)+O(\epsilon^3).
\end{align*}
\end{prop}
\vspace*{.1in}

\begin{proof}
In order to simplify the notation let
\[
\Delta_\epsilon:=f_+^{-1}\epsilon\Delta f_+^{-*}.
\]
Since $\Delta,f$ are both bounded, $|\trace(\Delta_\epsilon^k)|=O(\epsilon^k)$ as well as
$|\trace (\intpi \Delta_\epsilon \dtheta)^k|=O(\epsilon^k)$.
Using a Taylor series expansion,
\begin{align*}
\trace & \log \left(\intpi f_+^{-1}(f+\epsilon\Delta)f_+^{-*}\dtheta \right)\\
&=\trace\log \left(I+ \intpi \Delta_\epsilon\dtheta \right)\\
&=\trace\left(\intpi \Delta_\epsilon \dtheta\right)-\frac12 \trace \left(\intpi \Delta_\epsilon \dtheta\right)^2+O(\epsilon^3),
\end{align*}
while
\begin{align*}
\trace & \left( \intpi \log (f_+^{-1}(f+\epsilon\Delta)f_+^{-*})\dtheta\right)\\
&=\intpi \trace \log(I+\Delta_\epsilon)\dtheta\\
&=\intpi \trace (\Delta_\epsilon-\frac12 \Delta_\epsilon^2)\dtheta+O(\epsilon^3).
\end{align*}
Thus
\[
\D_2(f,f+\epsilon\Delta)=\frac12 \trace \left(\intpi \Delta_\epsilon^2 d\theta-\big(\intpi \Delta_\epsilon \dtheta\big)^2\right)+O(\epsilon^3).
\]
\end{proof}

Evidently, $\g_{2,f}$ and $\g_{1,f}$ are closely related.
The other choices of $\D$ similarly yield either $\g_{1,f}$, as noted earlier, or $\g_{2,f}$. In fact, $\g_{2,f}$ can be derived based on \eqref{eq:options}.
%Further, it is interesting to note the ``variance-looking'' nature of $\g_{2,f}$: if we denote
%\begin{eqnarray*}
%    \|\phi\|_{\var}^2&:=&\trace \left(\intpi\phi(\theta)^2\dtheta-\left(\intpi \phi(\theta)\dtheta\right)^2 \right).
%%    &=&\trace \intpi(\phi(\vartheta)-\intpi\phi(\theta)\dtheta )^2\frac{\vartheta}{2\pi}.
%\end{eqnarray*}
%Then
%\begin{equation}\label{g1b}
%\g_{2,f}(\Delta)=\|f_{+}^{-1}\Delta f_{-}^{-1}\|_{\var}^2.
%\end{equation}

We remark a substantial difference between $\g_{1,f}$ and $\g_{2,f}$. In contrast to $\g_{2,f}$, evaluation of $\g_{1,f}$ does not require computing $f_+$. However, on the other hand, both $\g_{1,f}$, and $\g_{2,f}$ are similarly unaffected by consistent scaling of $f$ and $\Delta$.

\section{Geometry on positive matrices}\label{sec:geometryonmatrices}

As indicated earlier, a Riemannian metric $\g(\Delta)$ on the space of Hermitian $m\times m$ matrices is a family of quadratic forms originating from inner products that depend smoothly on the Hermitian ``foot point'' $M$ ---the standard Hilbert-Schmidt metric $\g_{\rm HS}(\Delta)=\langle \Delta,\Delta\rangle :=\trace(\Delta^2)$ being one such. Of particular interest are metrics on the space of positive definite matrices that ensure the space is complete and geodesically complete\footnote{A space is complete when Cauchy sequences converge to points in the space. It is geodesically complete when the definition domain of geodesics extends to the complete real line $\mR$; i.e., extrapolating the path beyond the end points remains always in the space.}. For our purposes, matrices typically represent covariances. To this end a standard recipe for constructing a Riemannian metric is to begin with an information potential, such as the Boltzmann entropy of a Gaussian distribution with zero mean and covariance $M$,
\[
S(M):=-\frac{1}{2}\log(\det(M)) + {\rm constant},
\]
and define an inner product via its Hessian
\begin{eqnarray*}
\langle X,Y\rangle_M&:=&\frac{\partial^2}{\partial x\partial y}S(M+xX+yY)\large|_{x=0,y=0}\\
&=&\trace (M^{-1}XM^{-1}Y).
\end{eqnarray*}
The Riemannian metric so defined,
\begin{eqnarray*}
\g_M(\Delta):&=&\trace (M^{-1}\Delta M^{-1}\Delta)\\
&=&\|M^{-\frac12}\Delta M^{-\frac12}\|_{\rm Fr}^2,
\end{eqnarray*}
is none other than the Fisher-Rao metric on Gaussian distributions expressed in the space of the corresponding covariance matrices.

The relationship of the Fisher-Rao metric on Gaussian distributions with the metric $\g_{1,f}$ in \eqref{MetricDef1b} is rather evident. Indeed, $\g_M$ coincides with $\g_{1,f}$ for 
power spectra which are constant across frequencies, i.e., taking $f=M$ to be a constant Hermitian positive definite matrix.

It is noted that $g_M(\Delta)$ remains invariant under congruence, that is, 
\[
\g_{M}(\Delta)=\g_{TMT^*}(T\Delta T^*)
\]
for any square invertible matrix-function $T$.
This is a natural property to demand since it implies that the distance between covariance matrices
does not change under coordinate transformations. 
The same is inherited by $\g_{1,f}$ for power spectra.
It is for this reason that $\g_M$ has in fact been extensively studied in the context of general $C^*$-algebras and their positive elements; we refer to \cite[pg. 201-235]{Bhatia_positive} for a nice exposition of relevant material and for further references. Below we highlight certain key facts that are relevant to this paper. But first, and for future reference, we recall a standard result in differential geometry.

\begin{prop}\label{Prop:EquiGeodesic}
Let $\cM$ be a Riemannian manifold with $\|\Delta\|^2_M$ denoting the Riemannian metric at $M\in \cM$ and $\Delta$ a tangent direction at $M$.
For each pair of points $M_0$, $M_1\in\cM$ consider the path space
\begin{align*}
\Theta_{M_0,M_1}&:=\{M_\tau: [0, 1]\rightarrow \cM : ~M_\tau \text{~is a piecewise smooth} \\
 & \hspace*{25pt}\text{path connecting the two given points}\}.
\end{align*}
Denote by $\dot M_\tau := dM_\tau/d\tau$.
The arc-length
\begin{align*}
\int_0^1 \|\dot M_\tau\|_M d\tau,
\end{align*}
as well as the ``action/energy'' functional
\begin{align*}
\int_0^1  \|\dot M_\tau\|_M^2 d\tau
\end{align*}
attain a minimum at a common path in $\Theta_{f_0,f_1}$. Further, the minimal value of the arclength is the square root of the minimal value of the energy functional, and on a minimizing path the ``speed''
$\|\dot M_\tau\|_M$
remains constant for $\tau\in[0,1]$.
\end{prop}
\begin{proof}
See \cite[pg. 137]{Petersen_Riemannian}.
\end{proof}

\vspace*{.1in}
The insight behind the statement of the proposition is as follows.  The arclength is evidently unaffected by a re-parametrization of a geodesic connecting the two points. The ``energy'' functional on the other hand, is minimized for a specific parametrization of geodesic where the velocity stays constant. Thus, the two are intimately related. The proposition will be applied first to paths between matrices, but in the next section it will also be invoked for geodesics between power spectra.

Herein we are interested in geodesic paths $M_\tau$, $\tau\in[0,1]$, connecting positive definite matrices $M_0$ to $M_1$ and in computing the corresponding geodesic distances
 \[
\dist_{\g} (M_0,M_1)=\int_{0}^{1} \| M_\tau^{-1/2}\frac{dM_\tau}{d\tau}M_\tau^{-1/2}\|_{\rm Fr} d\tau.
 \]
Recall that a geodesic $M_\tau$ is the shortest path on the manifold connecting the beginning to the end.

\begin{thm}\label{Thm:GeodesicPD}
Given Hermitian positive matrices $M_0,M_1$ the geodesic between them with respect to $\g_M$ is unique (modulo re-parametrization) and given by
\begin{equation}\label{eq:Mgeodesic}
M_\tau=M_0^{1/2}(M_0^{-1/2}M_1M_0^{-1/2})^\tau M_0^{1/2},
\end{equation}
for $0\leq \tau \leq 1$. Further, it holds that
\[
\dist_{\g} (M_0, M_\tau)=\tau \dist_{\g}(M_0, M_1),\mbox{ for } \tau \in[0, 1],
\]
and the geodesic distance is
\[
\dist_{\g} (M_0, M_1)=\|\log (M_0^{-1/2}M_1M_0^{-1/2}) \|_{\rm Fr}.
\]
\end{thm}
\vspace*{.1in}

\begin{proof} A proof is given in  \cite[Theorem 6.1.6, pg.\ 205]{Bhatia_positive}.  However, since this is an important result for our purposes and for completeness, we provide an independent short proof relying on Pontryagin's minimum principle.

We first note that, since $\g_M$ is congruence invariant, the path $TM_\tau T^*$ is a geodesic between $TM_0 T^*$ and $TM_1 T^*$, for any invertible matrix $T$. Further, the geodesic length is independent of $T$. Thus, we set
\[
T=M_0^{-\frac12},
\]
and seek a geodesic path between
\begin{equation}\label{eq:boundary}
X_0=I \mbox{ and }X_1=M_0^{-\frac12} M_1 M_0^{-\frac12}.
\end{equation}
Appealing to Proposition \ref{Prop:EquiGeodesic} we seek
\begin{eqnarray}\label{eq:optimalcontrol}
&&\min\{ \int_0^1\trace(X_\tau^{-1}U_\tau X_\tau^{-1}U_\tau)d\tau,\\
&&\hspace*{20pt}\mbox{ subject to } \dot X_\tau=U_\tau, \mbox{ and } X_0,X_1\mbox{ specified}\}.\nonumber
\end{eqnarray}
Now, \eqref{eq:optimalcontrol} is a standard optimal control problem. The value of the optimal control must
annihilate the variation of the Hamiltonian with respect to the ``control'' $U_\tau$
\[
\trace(X_\tau^{-1}U_\tau X_\tau^{-1}U_\tau) + \trace(\Lambda_\tau U_\tau).
\]
Here, $\Lambda_\tau$ represents the co-state (i.e., Lagrange multiplier functions). The variation is
\[
\trace(2X_\tau^{-1}U_\tau X_\tau^{-1}\delta_U + \Lambda_\tau \delta_U)
\]
and this being identically zero for all $\delta_U$  implies that
\begin{equation}\label{Uopt}
U_\tau = -\frac12 X_\tau \Lambda_\tau X_\tau.
\end{equation}
Similarly, the co-state equation is obtained by considering the variation with respect to $X$. This gives
\[
\dot \Lambda_\tau = 2 X_\tau^{-1}U_\tau X_\tau^{-1}U_\tau X_\tau^{-1}.
\]
Substitute the expression for $U_\tau$ into the state and the co-state equations to obtain
\begin{eqnarray*}
\dot X_\tau &=& -\frac12 X_\tau \Lambda_\tau X_\tau\\
\dot \Lambda_\tau &=& \phantom{-}\frac12 \Lambda_\tau X_\tau \Lambda_\tau.
\end{eqnarray*}
Note that
\[
\dot X_\tau \Lambda_\tau+ X_\tau \dot \Lambda_\tau =0,
\]
identically, for all $\tau$. Hence, the product $X_\tau\Lambda_\tau$ is constant. Set
\begin{equation}\label{XLambda}
X_\tau\Lambda_\tau = - 2C.
\end{equation}
The state equation becomes
\[
\dot X_\tau = CX_\tau.
\]
The solution with initial condition $X_0=I$ is
\[
X_\tau= \exp(C\tau).
\]
Matching \eqref{eq:boundary} requires that $\exp(C)=X_1=M_0^{-\frac12} M_1 M_0^{-\frac12}$.
Thus, $X_\tau=(M_0^{-\frac12} M_1 M_0^{-\frac12})^\tau$ and the geodesic is as claimed.
Further,
\[
C=\log(M_0^{-\frac12} M_1 M_0^{-\frac12})
\]
while $U_\tau=CX_\tau$ from \eqref{XLambda} and \eqref{Uopt}.
So finally, for the minimizing choice of $U_\tau$ we get that the cost
\begin{eqnarray*}
\int_0^\tau\trace(X_\tau^{-1}U_\tau X_\tau^{-1}U_\tau)d\tau &=& \int_0^\tau\trace(C^2)d\tau\\
&=& \tau\|\log (M_0^{-1/2}M_1M_0^{-1/2}) \|_{\rm Fr}^2
\end{eqnarray*}
as claimed.\end{proof}

\begin{remark} It's important to point out the lower bound
\begin{align}\label{eq:loginequ}
\dist_{\g}(M_0, M_1)\geq \|\log M_0-\log M_1 \|_{\rm Fr}
\end{align}
on the geodesic distance which holds with equality when $M_0$ and $M_1$ commute. 
This is known as the exponential metric increasing property \cite[page 203]{Bhatia_positive} and will be used later on.
$\Box$
\end{remark}

The mid point of the geodesic path in \eqref{eq:Mgeodesic} is what is known as the geometric mean of the two matrices $M_0$ and $M_1$. This is commonly denoted by
\[
M_{\frac12}:=M_0\sharp M_1.
\]
Similar notation, with the addition of a subscript $\tau$, will be used to designate the complete geodesic path
\[
 M_\tau=M_0\sharp_\tau M_1:=M_0^{1/2}(M_0^{-1/2}M_1M_0^{-1/2})^\tau M_0^{1/2}
 \]
(see \cite{Bhatia_positive}).
A number of useful properties can be easily verified:\\
i) Congruence invariance: for any invertible matrix $T$,
\[
\dist_{\g}(M_0, M_1)=\dist_{\g}(T M_0 T^*, T M_1 T^*).
\]
ii) Inverse invariance:
\[
\dist_{\g}(M_0, M_1)=\dist_{\g}(M_0^{-1}, M_1^{-1}).
\]
iii) The metric satisfies the semiparallelogram law.\\
iv) The space of positive definite matrices metrized by $\dist_{\g}$ is complete; that is, any Cauchy sequence of positive definite matrices converges to a positive definite matrix.\\
v) Given any three ``points'' $M_0$, $M_1$, $M_2$,
\[
\dist_{\g}(M_0\sharp_\tau M_1, M_0\sharp_\tau M_2)\leq \tau \dist_{\g}(M_1, M_2),
\]
which implies that geodesics diverge at least as fast as ``Euclidean geodesics''.

\begin{remark}
Property v) implies that the Riemannian manifold of positive definite matrices with metric $\dist_\g$ has nonpositive sectional curvature \cite[pg. 39--40]{Jost_nonpositive}. %; that is, any sectional curvature defined by the curvature tensor on a tangent plane spanned by two linearly independent tangent vectors is nonpositive.
The nonpositive sectional curvature of a simply connected complete Riemannian manifold has several important geometric consequences. It implies the existence and uniqueness of a geodesic connecting any two points on the manifold \cite[pg. 3--4]{Jost_nonpositive}. Convex sets on such a manifold are defined by the requirement that geodesics between any two points in the set lie entirely in the set \cite[pg. 67]{Jost_nonpositive}. Then, ``projections'' onto the set exist in that there is always a closest point within convex set to any given point. Evidently, such a property should be valuable in applications, such as  speaker identification or speech recognition based on a database of speech segments; e.g., models may be taken as the ``convex hull'' of prior sample spectra and the metric distance of a new sample compared to how far it resides from a given such convex set.
Another property of such a manifold is that the center of mass of a set of points is contained in the closure of its convex hull \cite[pg. 68]{Jost_nonpositive}; this property has been used to define the geometric means of symmetric positive matrices in \cite{Moakher_differential}. $\Box$
\end{remark}

\section{Geodesics and geodesic distances}\label{sec:distances}

Power spectral densities are families of Hermitian matrices parametrized by the frequency $\theta$, and as such, can be thought of as positive operators on a Hilbert space. 
Geometries for positive operators have been extensively studied for some time now, and power spectral densities may in principle be studied with similar tools. However, what it may be somewhat surprising is that the geometries obtained earlier, based on the innovations flatness and optimal prediction, have points of contact with this literature. This was seen in the correspondence between the metrics that we derived.

In the earlier sections we introduced two metrics, $\g_1$ and $\g_2$.
Although there is a close connection between the two, as suggested by \eqref{g1a}, it is only for the former that we are able to identify
geodesics and compute the geodesic lengths, based on the material in Section \ref{sec:geometryonmatrices}.
We do this next.

%%%%%%%%
\begin{thm}
There exists a unique geodesic path $f_\tau$ with respect to $\g_{1,f}$, connecting any two spectra $f_0$, $f_1\in \cF$. The geodesic path %, parameterized by arc length,
is
\begin{align}\label{eq:geodesic2}
f_\tau=f_0^{1/2}(f_0^{-1/2}f_1f_0^{-1/2})^\tau f_0^{1/2},
\end{align}
for $0\leq \tau \leq 1$. The geodesic distance is
\[
\dist_{\g_1}(f_0, f_1)=\sqrt{\intpi \|\log f_0^{-1/2}f_1f_0^{-1/2} \|_{\rm Fr}^2 \dtheta}.
\]
\end{thm}

\vspace*{.1in}
\begin{proof}
As before,  in view of Proposition \ref{Prop:EquiGeodesic}, instead of the geodesic length we may equivalently consider minimizing the energy/action functional
\begin{align*}
{\rm E}&=\int_0^1 \intpi \|f_\tau^{-1/2}\dot{f}_\tau f_\tau^{-1/2}\|_{\rm Fr}^2 \dtheta d\tau\\
&= \intpi \int_0^1\|f_\tau^{-1/2}\dot{f}_\tau f_\tau^{-1/2}\|_{\rm Fr}^2  d\tau \dtheta.
\end{align*}
Clearly, this can be minimized point-wise in $\theta$ invoking Theorem \ref{Thm:GeodesicPD}.
Now, inversion as well as the fractional power of symmetric (strictly) positive matrices represent continuous and differentiable maps.
Hence, it can be easily seen that, because $f_0,f_1$ are in $\cF$ so is
\[
f_\tau=f_0^{1/2}(f_0^{-1/2}f_1f_0^{-1/2})^\tau f_0^{1/2}.
\]
Therefore, this path is the sought minimizer of
\[
\int_0^1\|f_\tau^{-1/2}\dot{f}_\tau f_\tau^{-1/2}\|_{\rm Fr}^2  d\tau
\]
and the geodesic length is as claimed.
\end{proof}

%We shall study the property of the geodesic distance $d_f(f_0, f_1)$ in a less specified space $\hat{\cF}_m$.
%It's straightforward to verify that $d_f(f_0, f_1)$ is both congruence and inverse invariant.  It also has the following properties:
\begin{cor}
Given any  $f_0$, $f_1$, $f_2\in \cF$, the function $\dist_{\g_1}(f_0\sharp_\tau f_1, f_0\sharp_\tau f_2)$ is convex on $\tau$.
\end{cor}

\vspace*{.1in}
\begin{proof}
The proof is a direct consequence of the convexity of the metric $\dist_g(\cdot, \cdot)$.\end{proof}

The importance of the statement in the corollary is that the metric space has nonpositive curvature.
Other properties are similarly inherited. For instance, $\dist_{\g_1}$ satisfies the semi-parallelogram law.

Next we explain that the closure of the space of positive differentiable power spectra, under $\g_1$, is simply power spectra that are squarely log integrable. This is not much of a surprise in view of the metric and the form of the geodesic distance. Thus, the next proposition shows that the completion, denoted by ``bar,'' is in fact

\begin{eqnarray}\nonumber
\bar{\cF}&:= &\{f\;\mid \mbox{$m\times m$ positive definite a.e.,}\\
&&\phantom{x}\mbox{on }[-\pi,\pi], ~\log{f}\in L_2[-\pi, \pi]\}. \label{eq:closure}
\end{eqnarray}
It should be noted that the metric $\dist_{\g_1}$ is not equivalent to an $L_2$-based metric
$\|\log(f_1)-\log(f_2)\|_2$ for the space. Here,
\[ \| h\|_2:=\sqrt{\intpi \|h\|_{\rm Fr}^2\dtheta}.\]
In fact, using the latter $\bar{\cF}$ has zero curvature while, using $\dist_{g_1}$, $\bar{\cF}$ becomes a space with non-positive (non-trivial) curvature.

\begin{prop}\label{prop:GeodesicComplete}
The completion of $\cF$ under $\dist_{\g_1}$ is as indicated in \eqref{eq:closure}.
\end{prop}

\vspace*{.1in}
\begin{proof} Clearly, for $f\in\cF$, $\log f\in L_2[-\pi,\pi]$ since $f$ is continuous on the closed interval and positive definite. Further, the logarithm maps positive differentiable matrix-functions to positive differentiable ones, bijectively.
Our proof of $\bar{\cF}$ being the completion of $\cF$ is carried out in three steps. First we will show that the limit of every Cauchy sequence in $\cF$ belongs to $\bar{\cF}$. Next we argue that every point in $\bar{\cF}$ is the limit of a sequence in $\cF$, which together with the first step shows that $\cF$ is dense in $\bar{\cF}$. Finally, we need to show that $\bar{\cF}$ is complete with $\dist_{\g_1}$.

First, consider a Cauchy sequence $\{f_n\}$ in $\cF$ which converges to $f$. Hence, there exists an $N$, such that for any $k\geq N$, $\dist_{\g_1}(f_k, f)<1$. Using the triangular inequality for $\dist_{\g_1}$, we have that
\[
\dist_{\g_1}(I, f)\leq \dist_{\g_1}(I, f_N)+\dist_{\g_1}(f_N, f),
\]
or, equivalently,
\[
\| \log{f} \|_2 < \| \log{f_N} \|_2+1.
\]
Since $ \| \log{f_N} \|_2$ is finite, $f \in \bar{\cF}$.  

Next, for any point $f$ in $\bar{\cF}$ which is not continuous, we show that it is the limit of a sequence in $\cF$. Let $h=\log f$, then $h\in L_2[-\pi, \pi]$. Since the set of differentiable functions $C^1[-\pi, \pi]$ is dense in $L_2[-\pi, \pi]$, there exits a sequence $\{h_n\in C^1[-\pi, \pi]\}$ which converges to $h$ in the $L_2$ norm. Using Theorem 3 in \cite[pg. 86]{Kolmogorov_FunctionsV2}, there exists a subsequence $\{h_{n_k}\}$ which converges to $h$ almost everywhere in $[-\pi, \pi]$, i.e.,
\[
\| h_{n_k}(\theta)-h(\theta)||_{\rm Fr}\rightarrow 0~~\text{a.e.,~~as~~} n_k\rightarrow \infty.
\]
Since the exponential map is continuous \cite[pg. 430]{Horn_Topics}, $\| e^{h_{n_k}(\theta)}-e^h(\theta)||_{\rm Fr}$ converges to $0$ almost everywhere as well. Using the sub-multiplicative property of the Frobenius norm, we have that
\[
\| I-e^{-h(\theta)}e^{h_{n_k}(\theta)}\|_{\rm Fr}\leq \|e^{-h(\theta)}\|_{\rm Fr} \| e^{h_{n_k}(\theta)}-e^h(\theta)\|_{\rm Fr},
\]
where the right side of the above inequality goes to zero. Thus the spectral radius of $( I-e^{-h(\theta)}e^{h_{n_k}(\theta)})$ goes to zero \cite[pg. 297]{Horn_MatrixAnalysis}. Hence,
all the eigenvalues $\lambda_i(e^{-h(\theta)}e^{h_{n_k}(\theta)})$, $1\leq i \leq m$, converge to $1$ as $k\to \infty$. Then, $f_{n_k}=e^{h_{n_k}}\in\cF$ and
\begin{align*}
\dist_{\g_1}(f_{n_k}, f)&=\sqrt{\intpi \|\log f^{-1/2}f_{n_k} f^{-1/2} \|_{\rm Fr}^2 \dtheta} \\
&=\sqrt{\intpi \sum_{i=1}^m \log ^2 \lambda_i(f^{-1}f_{n_k}) \dtheta} \\
&=\sqrt{\intpi \sum_{i=1}^m \log ^2 \lambda_i(e^{-h}e^{h_{n_k}}) \dtheta}.
\end{align*}
Since $\log \lambda_i(e^{-h}e^{h_{n_k}})\to 0$ a.e., for $1\leq i \leq m$, $\dist_{\g_1}(f_{n_k}, f)\to 0$ as well. Therefore, $f$ is the limit of $\{ f_{n_k}\}$.

Finally we show that $\bar{\cF}$ is complete under $\dist_{\g_1}$.
Let $\{f_n\}$ be a Cauchy sequence in $(\bar{\cF}, \dist_{\g_1})$, and let $h_n=\log f_n$.
Using the inequality (\ref{eq:loginequ}), we have
\[
\dist_{\g_1}(f_k, f_l)\geq \sqrt{\intpi \|h_k-h_l\|_{\rm Fr}^2\dtheta}.
\]
Thus $\{h_n\}$ is also a Cauchy sequence in $L_2[-\pi, \pi]$, which is a complete metric space. As a result, $\{h_n\}$ converges to a point $h$ in $L_2[-\pi, \pi]$. Following the similar procedure as in the previous step, there exists a subsequence $\{ f_{n_k}\}$ which converges to $f=e^h\in \bar{\cF}$. This completes our proof.
\end{proof}

\begin{remark} Geodesics of $\g_{2,f}$ for scalar power spectra were constructed in \cite{Georgiou_distance}. At the present time, a multivariable generalization appears to
be a daunting task. The main obstacle is of course non-commutativity of matricial density functions and the absence of an integral representation of analytic spectral factors in terms of matrix-valued power spectral densities.
In this direction we point out that some of the needed tools are in place. For instance, a square matrix-valued function which is analytic and non-singular in the unit disc $\mD$, admits a logarithm which is also analytic in $\mD$. To see this, consider such a matrix-function, say $f_+(z)$. The matrix logarithm is well defined locally in a neighborhood of any $z_0\in \mD$ via the Cauchy integral
\begin{equation}\nonumber
g(z) =\frac{1}{2\pi i } \int_{L_{z_0}} \ln(\zeta) (\zeta I-f_+(z))^{-1}d\zeta.
\end{equation}
Here, $L_{z_0}$ is a closed path in the complex plane that encompasses all of the eigenvalues of $f_+(z_0)$ and does not separate the origin from the point at $\infty$. The Cauchy integral gives a matrix-function $g(z)$ which is analytic in a sufficiently small neighborhood of $z_0$ in the unit disc $\mathbb D$ ---the size of the neighborhood being dictated by the requirement that the eigenvalues stay within $L_{z_0}$, and $\exp(g(z))=f_+(z)$. To define the logarithm consistently over $\mD$ we need to ensure that we always take the same principle value. This is indeed the case if we extend $g(z)$ via analytic continuation: since $f_+(z)$ is not singular anywhere in $\mD$ and the unit disc is simply connected, the values for $g(z)$ will be consistent, i.e., any path from $z_0$ to an arbitrary $z\in\mD$ will lead to the same value for $g(z)$. Thus, one can set $\log(f_+)=g$ and understand this to be a particular version of the logarithm.
Similarly, powers of $f_+$ can also be defined using Cauchy integrals,
\begin{equation}\nonumber
\frac{1}{2\pi i } \int_{L_{z_0}} \zeta^\tau (\zeta I-f_+(z))^{-1}d\zeta
\end{equation}
for $\tau\in[0,1]$, first in a neighborhood of a given $z_0\in\mD$, and then by analytic continuation to the whole of $\mD$. As with the logarithm, there may be several versions. Geodesics for $\g_{2,f}$ appear to be require paths in the space of cannonical spectral factors for the corresponding matricial densities, such as
$f_{\tau +}=f_{0+}(f_{0+}^{-1}f_{1+})_{+}^\tau$. However, the correct expression remains elusive at present. $\Box$
\end{remark}

\section{Examples}\label{sec:example}
We first demonstrate geodesics connecting two power spectral densities that correspond to all-pole models, i.e., two autoregressive (AR) spectra.
The geodesic path between them does not consist of AR-spectra, and it can be considered as a non-parametric model for the transition. The choice of AR-spectra for the end points
is only for convenience. As discussed earlier, the aim of the theory is to serve as a tool in non-parametric estimation, path following, morphing, etc., in the spectral domain.

\subsection*{A scalar example:}
Consider the two power spectral denisities
\[f_i(\theta)=\frac{1}{|a_i(e^{\jj\theta})|^2}, \;i\in \{0,1 \},\]
where
\begin{align*}
a_0=&(z^2-1.96\cos(\frac{\pi}{5})+0.98^2)(z^2-1.7\cos(\frac{\pi}{3})+0.85^2)\\
&(z^2-1.8\cos(\frac{2\pi}{3})+0.9^2),\\
a_1=&(z^2-1.96\cos(\frac{2\pi}{15})+0.98^2)(z^2-1.5\cos(\frac{7\pi}{30})+0.75^2)\\
&(z^2-1.8\cos(\frac{5\pi}{8})+0.9^2).
\end{align*}
Their roots are marked by $\times$'s and $\circ$'s respectively, in Figure \ref{fig:locus}, and shown with respect to the unit circle in the complex plane.
We consider and compare the following three ways of interpolating power spectra between $f_0$ and $f_1$.
\begin{figure}[htb]\begin{center}
\includegraphics[totalheight=4cm]{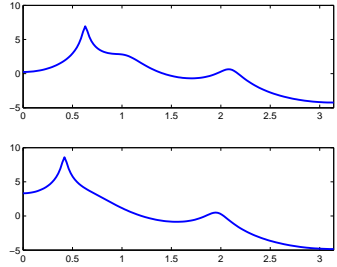}
\caption{Plots of $\log f_0(\theta)$ (upper) and $\log f_1(\theta)$ (lower) for $\theta\in [0, \pi]$.}\label{fig:spectra}\end{center}
\end{figure}
\begin{figure}[htb]\begin{center}
\includegraphics[totalheight=4cm]{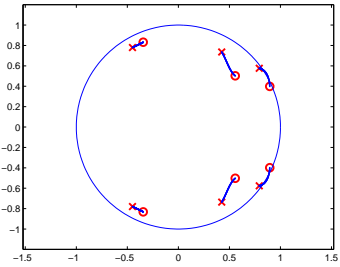}
\caption{Locus of the roots of $a_\tau(z)$ for $\tau\in[0,1]$.}\label{fig:locus}\end{center}
\end{figure}

First, a parametric approach where the AR-coefficient are interpolated:
\begin{subequations}
\begin{align}
&f_{\tau,\AR}(\theta)=\frac{1}{|a_\tau(e^{\jj\theta})|^2}, \label{eq:ar}
\end{align}
with $a_\tau(z)=(1-\tau)a_0(z)+\tau a_1(z)$. Clearly, there is a variety of alternative options (e.g., to interpolate partial reflection coefficients, etc.). However, our choice is intended to
highlight the fact that in a parameter space, admissible models may not always form a convex set. This is evidently the case here as the path includes factors that become ``unstable.''
The locus of the roots of $a_\tau(z)=0$ for $\tau\in[0,1]$ is shown in Figure \ref{fig:locus}.

Then we consider a linear segment connecting the two spectra:
\begin{align}
&f_{\tau,\linear}=(1-\tau)f_0+\tau f_1. \label{eq:linear}
\end{align}
Again, this is to highlight the fact that the space of power spectra is not linear, and in this case, extrapolation beyond the convex linear combination of the two spectra leads to inadmissible function (as the path leads outside of the cone of positive functions).
Finally, we provide the $\g_1$-geodesic between the two
\begin{align}
&f_{\tau,\geodesic}=f_0(\frac{f_1}{f_0})^\tau. \label{eq:geo}
\end{align}
\end{subequations}
We compare $f_{\tau,\AR}$, $f_{\tau,\linear}$ and $f_{\tau,\geodesic}$ for $\tau\in\{\frac{1}{3}, \frac{2}{3}, \frac{4}{3}\}$.
We first note that in plotting $\log f_{\tau,\AR}$ in Figure \ref{fig:ar}, that  $f_{\frac{2}{3},\AR}$ is not shown since it is not admissible.
\begin{figure}[htb]\begin{center}
\includegraphics[totalheight=3cm]{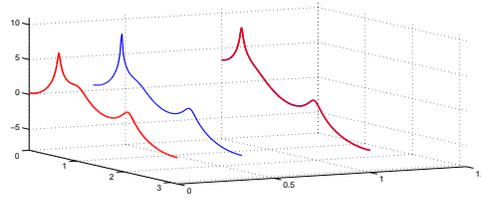}
\caption{$\log f_{\tau,\AR}(\theta)$ for $\tau=\frac{1}{3}, \frac{2}{3}, \frac{4}{3}$ (blue), $\tau=0,1$ (red).}\label{fig:ar}\end{center}
\end{figure}
Likewise $\log f_{\tau,\linear}$ in Figure \ref{fig:linear} breaks up for $\tau=\frac{4}{3}$, since $f_{\frac{4}{3},\linear}$ becomes negative for a range of frequencies --dashed curve indicates the absolute value of the logarithm when this takes complex values.
\begin{figure}[htb]\begin{center}
\includegraphics[totalheight=3cm]{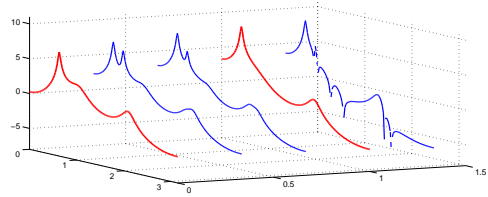}
\caption{$\log f_{\tau,\linear}(\theta)$ for $\tau=\frac{1}{3}, \frac{2}{3}, \frac{4}{3}$ (blue), $\tau=0,1$ (red).}\label{fig:linear}\end{center}
\end{figure}
The plot of $\log f_{\tau,\geodesic}$ is defined for all the $\tau$ and shown in Figure \ref{fig:geo}.
\begin{figure}[htb]\begin{center}
\includegraphics[totalheight=3cm]{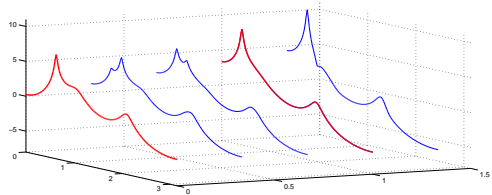}
\caption{$\log f_{\tau,\geodesic}(\theta)$ for $\tau=\frac{1}{3}, \frac{2}{3}, \frac{4}{3}$ (blue), $\tau=0,1$ (red).}\label{fig:geo}\end{center}
\end{figure}
It is worth pointing out how two apparent ``modes'' in $f_{\tau,\linear}$ and $f_{\tau,\geodesic}$ are swapping their dominance, which does not occur when following $f_{\tau,\AR}$.

\subsection*{A multivariable example:}

Consider the two matrix-valued power spectral densities
\begin{align*}
f_0=\left[
      \begin{array}{cc}
        1 & 0 \\
        0.1e^{\jj\theta} & 1 \\
      \end{array}
    \right]\left[
      \begin{array}{cc}
        \frac{1}{|a_0(e^{\jj\theta})|^2} & 0 \\
        0 & 1 \\
      \end{array}
    \right]\left[
      \begin{array}{cc}
        1 & 0.1e^{-\jj\theta}\\
        0 & 1 \\
      \end{array}
    \right]
\end{align*}

\begin{align*}
f_1=\left[
      \begin{array}{cc}
        1 & 0.1e^{\jj\theta} \\
        0 & 1 \\
      \end{array}
    \right]\left[
      \begin{array}{cc}
        1 & 0 \\
        0 & \frac{1}{|a_1(e^{\jj\theta})|^2} \\
      \end{array}
    \right]\left[
      \begin{array}{cc}
        1 & 0 \\
        0.1e^{-\jj\theta} & 1 \\
      \end{array}
    \right].
\end{align*}
Typically, these reflect the dynamic relationship between two time series; in turn these may represent noise input/output of dynamical systems or measurements across independent array of sensors, etc.
The particular example reflects the typical effect of an energy source shifting its signature from one of two sensors to the other as, for instance, a possible scatterer moves with respect to the two sensors.

Below $f_0$ and $f_1$ are shown in Fig.\ \ref{fig:F0} and Fig.\ \ref{fig:F1}, respectively.
Since the value of a power spectral density $f$, at each point in frequency, is a Hermitian matrix, our convention is to show in the (1,1), (1,2) and (2,2) subplots the log-magnitude of the entries $f(1,1), f(1,2)$ (which is the same as $f(2,1)$) and $f(2,2)$, respectively. Then, since only $f(1,2)$ is complex (and the complex conjugate of $f(2,1)$), we plot its phase in the (2,1) subplot.

\begin{figure}[htb]\begin{center}
\includegraphics[totalheight=5cm]{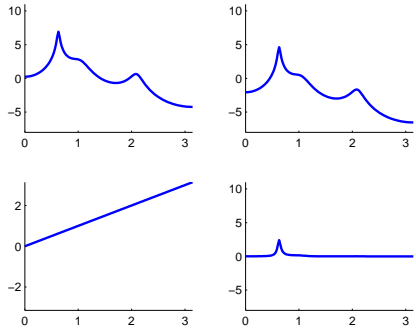}
\caption{Subplots (1,1), (1,2) and (2,2) show $\log f_0(1,1), \log |f_0(1,2)|$ (same as $\log|f_0(2,1)|$) and $\log f_0(2,2)$. Subplot (2,1) shows ${\rm arg}(f_0(2,1))$.}\label{fig:F0}\end{center}
\end{figure}

\begin{figure}[htb]\begin{center}
\includegraphics[totalheight=5cm]{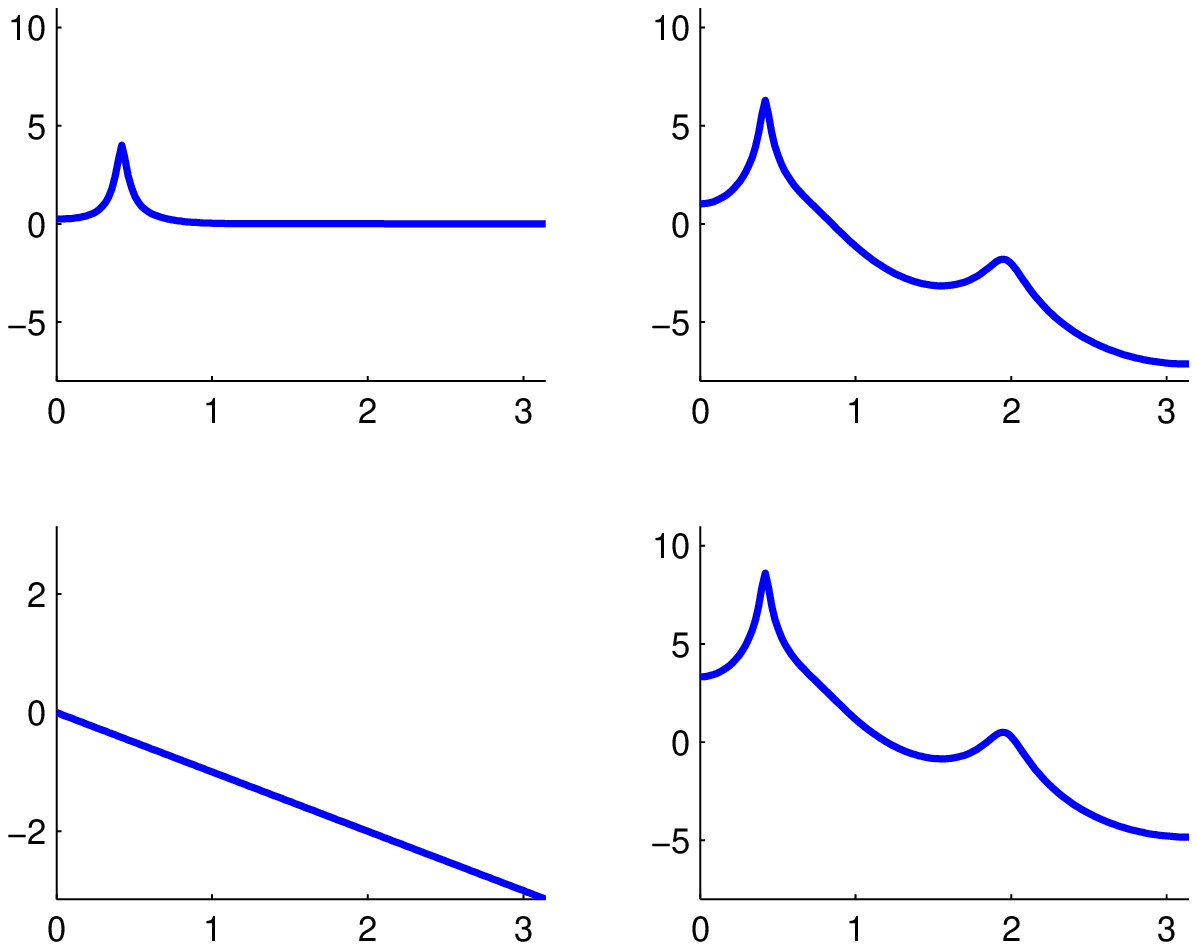}
\caption{Subplots (1,1), (1,2) and (2,2) show $\log f_1(1,1), \log |f_1(1,2)|$ (same as $\log|f_1(2,1)|$) and $\log f_0(2,2)$. Subplot (2,1) shows ${\rm arg}(f_1(2,1))$.}\label{fig:F1}\end{center}
\end{figure}

Three dimensional surface show the geodesic connecting $f_0$ to $f_1$ in Figure \ref{fig:matrix3d}. Here, $f_{\tau,\geodesic}$ is drawn using
\begin{align*}
f_{\tau,\geodesic}=f_0^{\frac{1}{2}}(f_0^{-\frac{1}{2}}f_1f_0^{-\frac{1}{2}})^\tau f_0^{\frac{1}{2}}.
\end{align*}
\begin{figure}[htb]\begin{center}
\hspace*{-10pt}\includegraphics[totalheight=8cm]{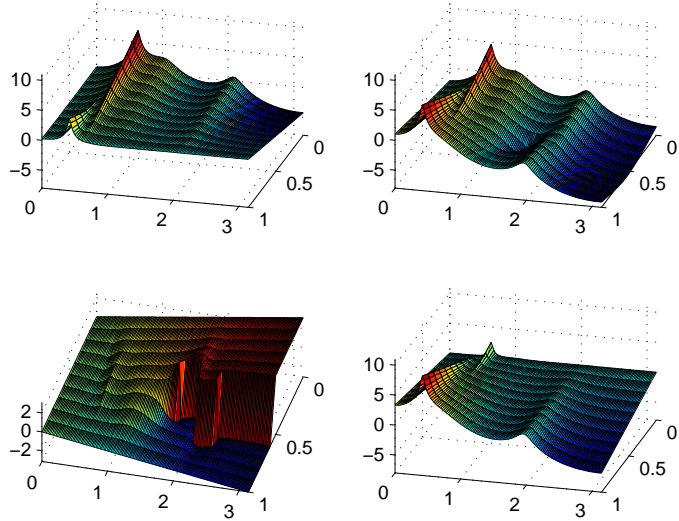}
\caption{Subplots (1,1), (1,2) and (2,2) show $\log f_\tau(1,1), \log |f_\tau(1,2)|$ (same as $\log|f_\tau(2,1)|$) and $\log f_\tau(2,2)$. Subplot (2,1) shows ${\rm arg}(f_\tau(2,1))$, for $\tau\in [0,1]$.}\label{fig:matrix3d}\end{center}
\end{figure}
It is interesting to observe the smooth shift of the energy across frequency and directionality.

\section{Conclusions}\label{sec:conclusion}
The aim of this study has been to develop multivariable divergence measures and metrics for matrix-valued power spectral densities. These are expected to be useful in quantifying uncertainty in the spectral domain, detecting events in non-stationary time series, smoothing and spectral estimation in the context of vector valued stochastic processes. The spirit of the work follows closely classical accounts going back to \cite{Basseville_distance,Gray_distortion} and proceeds along the lines of \cite{Georgiou_distance}. Early work in signal analysis and system identification has apparently focused only on divergence measures between scalar spectral densities, and only recently have such issues on multivariable power spectra attracted attention \cite{Ferrante_time,Ferrante_hellinger}.
Further, this early work on scalar power spectra
was shown to have deep roots in statistical inference, the Fisher-Rao metric, and Kullback-Leibler divergence  \cite{Kass_geometrical},  \cite[page 371]{Gray_distortion}, \cite{Georgiou_distance},  \cite{Yu_kullback}. Thus, it is expected that interesting connections between the geometry of multivariable power spectra and information geometry will be established as well.

\bibliographystyle{IEEEtran}
\bibliography{IEEEabrv,MPSD}
\end{document}